\pdfoutput=1
\RequirePackage{ifpdf}
\ifpdf 
\documentclass[pdftex]{sigma}
\else
\documentclass{sigma}
\fi

\usepackage{tikz-cd}

\numberwithin{equation}{section}

\newtheorem{Theorem}{Theorem}[section]
\newtheorem{thm}[Theorem]{Theorem}
\newtheorem*{Theorem*}{Theorem}

\newtheorem{cor}[Theorem]{Corollary}
\newtheorem{Lemma}[Theorem]{Lemma}

\newtheorem{prop}[Theorem]{Proposition}
 { \theoremstyle{definition}

\newtheorem{Example}[Theorem]{Example}

\newtheorem{rmk}[Theorem]{Remark} }

\newcommand{\BC}{{\mathbb{C}}}

\newcommand{\BQ}{{\mathbb{Q}}}

\newcommand{\BZ}{{\mathbb{Z}}}

\newcommand{\CE}{{\mathcal E}}
\newcommand{\CF}{{\mathcal F}}

\newcommand{\CL}{{\mathcal L}}

\newcommand{\CO}{{\mathcal O}}

\newcommand{\CZ}{{\mathcal Z}}

\newcommand{\Fg}{{\mathfrak{g}}}

\newcommand{\pt}{{\mathsf{p}}}
\newcommand{\ch}{{\mathrm{ch}}}
\newcommand{\td}{{\mathrm{td}}}

\DeclareFontFamily{OT1}{rsfs}{}
\DeclareFontShape{OT1}{rsfs}{n}{it}{<-> rsfs10}{}
\DeclareMathAlphabet{\curly}{OT1}{rsfs}{n}{it}

\newcommand\End{\operatorname{End}}

\newcommand{\Coh}{\mathrm{Coh}}
\newcommand{\Pic}{\mathop{\rm Pic}\nolimits}

\newcommand{\SO}{\mathrm{SO}}

\newcommand\rk{\operatorname{\mathsf{rk}}}

\begin{document}
\allowdisplaybreaks

\newcommand{\arXivNumber}{2201.03833}

\renewcommand{\thefootnote}{}

\renewcommand{\PaperNumber}{076}

\FirstPageHeading

\ShortArticleName{Universality of Descendent Integrals over Moduli Spaces}

\ArticleName{Universality of Descendent Integrals\\ over Moduli Spaces of Stable Sheaves on $\boldsymbol{K3}$ Surfaces\footnote{This paper is a~contribution to the Special Issue on Enumerative and Gauge-Theoretic Invariants in honor of Lothar G\"ottsche on the occasion of his 60th birthday. The~full collection is available at \href{https://www.emis.de/journals/SIGMA/Gottsche.html}{https://www.emis.de/journals/SIGMA/Gottsche.html}}}

\Author{Georg OBERDIECK}

\AuthorNameForHeading{G.~Oberdieck}

\Address{Mathematisches Institut, Universit\"at Bonn, Endenicher Allee 60, D-53115 Bonn, Germany}
\Email{\href{mailto:georgo@math.uni-bonn.de}{georgo@math.uni-bonn.de}}

\ArticleDates{Received January 23, 2022, in final form October 06, 2022; Published online October 13, 2022}

\Abstract{We interprete results of Markman on monodromy operators as a universality statement for descendent integrals over moduli spaces of stable sheaves on $K3$ surfaces. This yields effective methods to reduce these descendent integrals to integrals over the punctual Hilbert scheme of the $K3$ surface. As an application we establish the higher rank Segre--Verlinde correspondence for $K3$ surfaces as conjectured by G\"ottsche and Kool.}


\Classification{14D20; 14J28; 14J80; 14J60}

\renewcommand{\thefootnote}{\arabic{footnote}}
\setcounter{footnote}{0}

\section{Introduction}
\subsection{Descendent integrals}
Let $M$ be a proper and fine\footnote{See Remark~\ref{rmk:twisted universal sheaf} for the extension to the case where only a quasi-universal family exists.}
moduli space of Gieseker stable sheaves $F$ on a $K3$ surface $S$ with Mukai vector
\[
v(F) := \ch(F) \sqrt{\td_S} = v \in H^{\ast}(S,\BZ).
\]
Let $\pi_{M}$, $\pi_S$ be the projections of  $M \times S$ to the factors
and let $\CF \in \Coh(M \times S)$ be a universal family.
We define the $k$-th descendent of a class $\gamma \in H^{\ast}(S,\BQ)$ by
\begin{equation}
\tau_k(\gamma) = \pi_{M \ast}(\pi_S^{\ast}(\gamma) \ch_k(\CF)) \in H^{\ast}(M).
\label{defn descendent}
\end{equation}

Let $P(c_1, c_2, c_3, \dots )$ be a polynomial and consider
an arbitrary integral of descendents and Chern classes of the tangent bundle over the moduli space:
\begin{equation} \label{descendent integral}
\int_{M} \tau_{k_1}(\gamma_1) \cdots \tau_{k_{\ell}}(\gamma_{\ell}) P(c_r(T_{M})).
\end{equation}
The goal of this paper is to explain the following application of
Markman's work~\cite{Markman} on monodromy operators:
\begin{thm}
\label{thm:intro main}
Any integral of the form \eqref{descendent integral}
 can be effectively reconstructed from the set of all integrals \eqref{descendent integral}, where $M$ is replaced by the Hilbert scheme of $n$ points of a $K3$ surface, with $n=\dim M/2$.
\end{thm}

We refer to Section~\ref{sec:Markman universality} for the precise form the reconstruction of the theorem takes.
In particular, Theorem~\ref{thm:universality} is a universality statement for descendent integrals over $M$, that immediately implies Theorem~\ref{thm:intro main}.

\subsection{Segre numbers}
As a concrete application of Theorem~\ref{thm:intro main}
we prove a conjecture of G\"ottsche and Kool
which was made in~\cite[Conjecture~5.1]{GK}:
Consider the decomposition of $v \in H^{\ast}(S,\BZ)$ according to degree
\[
v = (\rk(v), c_1(v), v_2) \in H^0(S,\BZ) \oplus H^2(S,\BZ) \oplus H^4(S,\BZ),
\]
and assume that
\[
\rk(v) > 0.
\]

For any topological $K$-theory class $\alpha \in K(S)$ define
\[
\alpha_M = \ch\bigl(-\pi_{M \ast}\big(\pi_S^{\ast}(\alpha)\otimes\CF \otimes \det(\CF)^{-1/\rk(v)} \big)\bigr)
\]
whenever a $\rk(v)$-th root of $\det(\CF)$ exists.
Otherwise we define $\alpha_M$ by a formal application of the Grothendieck--Riemann--Roch formula.
Let $c(\alpha_M)$ be the Chern class corresponding to $\alpha_M$, see Remark~\ref{rmk: c applied to coh class}.

For $\sigma \in H^{\ast}(S)$ consider, with the same convention if the root does not exist, the class
\[
\mu_M(\sigma)=- \pi_{M \ast}\big(\ch_2\big(\CF \otimes \det(\CF)^{-1/\rk(v)}\big)\pi_S^{\ast}(\sigma)\big).
\]
We will usually drop the subscript $M$ from the notation.

\begin{thm} \label{thm:Intro}
Let $n = \frac{1}{2} \dim M$ and let $\pt \in H^4(S,\BZ)$ be the class of a point.
For any $\alpha \in K(S)$,
class $L \in H^2(S)$ and $u \in \BC$ we have
\begin{equation*}
\int_{M} c(\alpha_M) {\rm e}^{\mu(L) + u \mu(\pt)}
=\int_{S^{[n]}} c(\beta_{S^{[n]}}) {\rm e}^{\mu(L) + u \rk(v) \mu(\pt)},
\end{equation*}
where $\beta \in K(S)$ is any $K$-theory class such that
\begin{gather}
\rk(\beta) = \frac{\rk(\alpha)}{\rk(v)}, \quad\
c_1(\alpha)^2 = c_1(\beta)^2, \quad\
c_1(\alpha) \cdot L = c_1(\beta) \cdot L, \quad\
v_2(\beta) = \rk(v) v_2(\alpha).
\label{conditions on beta}
\end{gather}
\end{thm}

As explained in~\cite[Corollary~5.2]{GK} this implies the following closed evaluation
of the {\em Segre numbers} of $M$:

\begin{cor} \label{cor:Segre}
Let $\rho = \rk(v)$, $s = \rk(\alpha)$, $n=\frac{1}{2} \dim M$.
Then we have
\[
\int_{M} c(\alpha_M) = \mathrm{Coeff}_{z^n}\bigl(V_{s}^{c_2(\alpha)} W_{s}^{c_1(\alpha)^2} X_{s}^{2}  \bigr),
\]
where
\begin{gather*}
V_s(z) = \bigg(1+\bigg(1-\frac{s}{\rho}\bigg)t\bigg)^{1-s} \bigg(1+\bigg(2-\frac{s}{\rho}\bigg)t\bigg)^s \bigg(1+\bigg(1-\frac{s}{\rho}\bigg)t\bigg)^{\rho-1},
\\
W_s(z) = \bigg(1+\bigg(1-\frac{s}{\rho}\bigg)t\bigg)^{\frac{1}{2}s-1} \bigg(1+\bigg(2-\frac{s}{\rho}\bigg)t\bigg)^{\frac{1}{2}(1-s)} \bigg(1+\bigg(1-\frac{s}{\rho}\bigg)t\bigg)^{\frac{1}{2} - \frac{1}{2} \rho},
\\
X_s(z) =  \bigg(1+\bigg(1-\frac{s}{\rho}\bigg)t\bigg)^{\frac{1}{2} s^2-s} \bigg(1+\bigg(2-\frac{s}{\rho}\bigg)t\bigg)^{-\frac{1}{2}s^2+\frac{1}{2}}
\\ \hphantom{X_s(z) =}
{}\times \bigg(1+\bigg(1-\frac{s}{\rho}\bigg)\bigg(2-\frac{s}{\rho}\bigg)t\bigg)^{-\frac{1}{2}} \bigg(1+\bigg(1-\frac{s}{\rho}\bigg)t\bigg)^{-\frac{(\rho-1)^2}{2\rho} s}
\end{gather*}
under the variable change $z = t \big(1+\big(1-\tfrac{s}{\rho}\big) t\big)^{1-\frac{s}{\rho}}$.
\end{cor}

The Segre numbers of the Hilbert scheme of $n$ points on the $K3$ surface $S$
were determined by Marian, Oprea and Pandharipande~\cite{MOP}.
In particular, they found the series $V_s$, $W_s$, $X_s$.
All that Theorem~\ref{thm:Intro} does here is move their result from Hilbert schemes
to moduli spaces of sheaves of arbitrary rank.
Earlier work on Segre numbers can be found in~\cite{EGL, Lehn, MOP1, MOP2, Voisin}.

\subsection{Segre/Verlinde correspondence}
G\"ottsche and Kool conjectured that the Segre numbers of moduli spaces of stable sheaves on surfaces
are related by an explicit correspondence to the {\em Verlinde numbers} of these moduli spaces.
For $K3$ surfaces the Verlinde numbers are known explicitly by
\[
\chi\big(M, \mu(L) \otimes E^{\otimes r}\big) = \mathrm{Coeff}_{w^n}\big( G_{r}^{\chi(L)} \, F_{r}^{\frac{1}{2} \chi(\CO_S)} \big),
\]
where
\begin{equation*}
F_r(w) = (1+v)^{\frac{r^2}{\rho^2}}\bigg(1+\frac{r^2}{\rho^2} v\bigg)^{-1}, \qquad
G_r(w) = 1+v
\end{equation*}
under the variable change $w = v(1+v)^{r^2 / \rho^2-1},$
and we refer to~\cite[equation~(4)]{GK} for the definition of the class $\mu(L) \otimes E^{\otimes r} \in \Pic(M)_{\BQ}$.
The Verlinde numbers of the Hilbert schemes of points of~$K3$ surfaces (and in particular the series $F_r$, $G_r$) were first computed in~\cite{EGL}.
The computation for moduli spaces of higher rank sheaves reduces to the Hilbert scheme case as shown in~\cite{GNY2} using hyperk\"ahler geometry,
parallel to Theorem~\ref{thm:Intro}.

The functions $F_r$, $G_r$ and $V_s$, $W_s$, $X_s$ are related by the following variable change~\cite{GK}:
\begin{gather*}
F_r(w) = V_s(z)^{\frac{s}{\rho} (\rho^{\frac{1}{2}} - \rho^{-\frac{1}{2}})^2} \, W_s(z)^{-\frac{4s}{\rho}} \, X_s(z)^2,
\\
G_r(w) = V_s(z) \, W_s(z)^2,
\end{gather*}
where $s = \rho + r$ and $v = t\big( 1- \tfrac{r}{\rho} t \big)^{-1}$.

Hence with Corollary~\ref{cor:Segre} we have proven that the Segre and Verlinde numbers of moduli spaces of stable sheaves on $K3$ surfaces
are related by this variable change.
This is the $K3$ surface case of the higher-rank Segre--Verlinde correspondence conjectured by G\"ottsche--Kool~\cite[Conjecture~1.7]{GK}.

\begin{cor}
The higher-rank Segre--Verlinde correspondence holds for $K3$ surfaces.
\end{cor}

\subsection{Plan}
In Section~\ref{sec:Markman universality}, we use results from Markman's beautiful article~\cite{Markman}
to formulate a universality result for descendent integrals
of moduli spaces of stable sheaves on $K3$ surfaces, see Theorem~\ref{thm:universality}.
This immediately yields Theorem~\ref{thm:intro main}.
In Section~\ref{GK conjecture}, we prove Theorem~\ref{thm:Intro}.

\section{Markman's universality} \label{sec:Markman universality}
\subsection{Basic definitions} 
Let $S$ be a $K3$ surface and consider the lattice $\Lambda = H^{\ast}(S,\BZ)$ endowed with the Mukai pairing
\[
(x , y) := - \int_S x^{\vee} y,
\]
where, if we decompose an element $x \in \Lambda$ according to degree as $(r,D,n)$, we have written $x^{\vee} = (r,-D,n)$.
We will also write
\[
\rk(x) = r, \qquad c_1(x) = D, \qquad v_2(x) = n.
\]
Given a sheaf or complex $E$ on $S$ the Mukai vector of $E$ is defined by
\[
v(E) = \sqrt{\td_S} \cdot \ch(E) \in \Lambda.
\]

Let $v \in \Lambda$ be an effective\footnote{Following~\cite[Definition 1.1]{Markman},
this means that $v \cdot v \geq -2$ and $\rk(v) \geq 0$, and if $\rk(v)=0$ then $c_1(v)$ is effective or zero, and if $\rk(v)=c_1(v)=0$ then $v_2>0$.} vector, $H$ be an ample divisor on $S$
and let
\[
M := M_H(v)
\]
be the moduli space of $H$-stable sheaves with Mukai vector $v$.
The moduli space is smooth and holomorphic-symplectic of dimension $2+(v,v)$.
We further assume that the Mukai vector~$v$ is primitive, and the polarization $H$ is $v$-generic (see~\cite[Theorem~6.2.5]{HL}),
so that $M$ is also proper (in particular, semistability is equivalent to stability).
We also assume that there exists a universal sheaf $\CF$ on $M_H(v) \times S$.

\begin{rmk} \label{rmk:twisted universal sheaf}
The results we state below also hold in the case where there exists only a twisted universal sheaf.
More precisely, all statements below can be formulated
in terms of the Chern character $\ch(\CF)$ alone and this class can be defined in the twisted case as well, see~\cite[Section~3]{MarkmanGenerators}.
The proofs carry over likewise since all ingredients hold in the twisted case as well. 
\end{rmk}
\begin{rmk}
More generally, one can also work with $\sigma$-stable objects for a Bridgeland stability condition in the distinguished component.
\end{rmk}

Assume from now on that\footnote{We return to the case $\dim M = 2$ in Section~\ref{subsec:Case dim 2}.}
\[
\dim M = (v,v) + 2 > 2.
\]
Consider the morphism $\theta_{\CF} \colon \Lambda \to H^2(M_H(v),\BZ)$ defined by
\begin{equation}
\theta_{\CF}(x) = \big[ \pi_{M \ast}\big( \ch(\CF)
\pi_S^{\ast}\big( \sqrt{\td_S} \cdot x^{\vee}  \big) \big) \big]_{2},
\label{pushforward} \end{equation}
where $[ - ]_{k}$ stands for taking the degree $k$ component of a cohomology class.
Then $\theta_{\CF}$ restricts to an isomorphism
\begin{equation}
\theta = \theta_{\CF}|_{v^{\perp}} \colon \ v^{\perp} \xrightarrow{\cong} H^2(M_H(v),\BZ) \label{identification}
\end{equation}
which does not depend on the choice of universal family (use that the degree $0$ component of the pushforward in \eqref{pushforward} vanishes)
and for which we hence have dropped the subscript $\CF$.
The~isomorphism $\theta$ is an isometry with respect to the Mukai pairing on the left, and the pairing given by the Beauville--Bogomolov--Fujiki form on the right.
We will identify $v^{\perp} \subset \Lambda$ with $H^2(M_{H}(v),\BZ)$ under this isomorphism.

The universal sheaf $\CF$ and hence its Chern character $\ch(\CF)$ is uniquely determined only up to tensoring by the pullback of a line bundle from $M$.
Following~\cite{Markman}, we can pick a canonical normalization as follows:
\[
u_v := \exp\bigg( \frac{\theta_{\CF}(v) }{(v,v)} \bigg) \cdot \ch(\CF) \cdot \sqrt{\td_S} \ \in H^{\ast}(M \times S),
\]
where we have suppressed the pullback by the projections to $M$ and $S$ in the first and last term on the right.
We will follow similar conventions throughout.
It is immediate to check that $u_v$ is independent from the choice of universal family (replace $\CF$ by $\CF \otimes \pi_M^{\ast} \CL$ and calculate, see~\cite[Lemma~3.1]{Markman}).

\begin{Example} 
Let $M = S^{[n]}$ be the Hilbert scheme of $n$ points on $S$. We have $v=1-(n-1) \pt$,
and we always take $\CF = I_{\CZ}$, the ideal sheaf of the universal subscheme.
If $\alpha \in H^2(S)$ is the class of an effective divisor $A \subset S$, then
\[
\theta(\alpha) = \pi_{S^{[n]} \ast}\big(\ch_2(\CO_{\CZ}) \pi_S^{\ast}(\alpha)\big)
\]
is the class of the locus of subschemes incident to $A$.
If we denote
\[
\delta := -\frac{1}{2} \Delta_{S^{[n]}} = c_1( \pi_{S^{[n]} \ast} \CO_{\CZ} ) = \pi_{S^{[n]} \ast} \ch_3(\CO_{\CZ}),
\]
where $\Delta_{S^{[n]}}$ is the class of the locus of non-reduced subschemes, then
under the identification~\eqref{identification} we have
$\delta = - \big( 1 + (n-1) \pt \big)$.
Because $\theta_{\CF}(v)=-\delta$ the canonical normalization of~$\ch(\CF)$ takes the form
\[
u_v = \exp\bigg( \frac{-\delta}{2n-2} \bigg) \ch(I_{\CZ}) \sqrt{\td_S}.
\]
\end{Example}

\subsection{Markman's operator} 
For $i=1,2$ let $(S_i, H_i, v_i)$ be the data defining
proper fine moduli space of stable sheaves $M_i = M_{H_i}(S_i,v_i)$,
and let $\CF_i$ be the universal family on $M_i \times S_i$.
Consider an isometry of Mukai lattices
\[
g\colon\ H^{\ast}(S_1, \BZ) \to H^{\ast}(S_2, \BZ)
\]
such that $g(v_1) = v_2$.
Let $K(S)$ be the topological $K$-group of $S$ endowed with the Euler pairing $(E , F) = -\chi(E^{\vee} \otimes F)$. We identify $g$ with an isometry
\[
g\colon\ K(S_1) \to K(S_2)
\]
through the lattice isometry
$K(S) \xrightarrow{\cong} H^{\ast}(S,\BZ)$ given by $E \mapsto v(E)$.
Hence the following diagram commutes
\[
\begin{tikzcd}
K_{\mathrm{top}}(S_1) \ar{r}{g} \ar{d}{v} & K_{\mathrm{top}}(S_2) \ar{d}{v} \\
H^{\ast}(S_1, \BZ) \ar{r}{g} & H^{\ast}(S_2, \BZ).
\end{tikzcd}
\]
Similar identification will apply to morphisms $g$ defined over $\BC$.
The Markman operator associated to $g$ is given by the following result:

{\samepage\begin{thm}[Markman] \label{thm:Markman_operator} For any isometry $g\colon H^{\ast}(S_1, \BC) \to H^{\ast}(S_2, \BC)$ such that $g(v_1) = v_2$
there exists a unique operator
\[
\gamma(g)\colon \ H^{\ast}(M_1, \BC) \to H^{\ast}(M_2, \BC)
\]
such that
\begin{enumerate}
\item[$(a)$] $\gamma(g)$ is a degree-preserving isometric\footnote{We endow $H^{\ast}(M)$ with the Poincar\'e pairing: $\langle x,y \rangle = \int_{M} x y$ for all $x,y \in H^{\ast}(M)$.} ring-isomorphism,
\item[$(b)$] $(\gamma(g) \otimes g)(u_{v_1}) = u_{v_2}$.
\end{enumerate}
The operator is called the \emph{Markman operator} and given by
\begin{equation}
\gamma(g) = c_{\dim(M)}\bigl[ - \pi_{13 \ast}\big( \pi_{12}^{\ast} ((1 \otimes g) u_{v_1})^{\vee} \cdot \pi_{23}^{\ast} u_{v_2} \big) \bigr],
\label{Mopnew}
\end{equation}
where $\pi_{ij}$ is the projection of $M_1 \times S_2 \times M_2$ to the $(i,j)$-th factor.
Moreover, we have
\begin{enumerate}
\item[$(c)$] $\gamma(g_1) \circ \gamma(g_2) = \gamma(g_1g_2)$ and $\gamma(g)^{-1} = \gamma\big(g^{-1}\big)$ if it makes sense.
\item[$(d)$] $\gamma(g) c_k(T_{M_1}) = c_k(T_{M_2})$.
\end{enumerate}
\end{thm}}

\begin{rmk} \label{rmk: c applied to coh class}
Here the Chern class $c_m$ in \eqref{Mopnew} has the following definition:
Let
\[
\ell \colon \ \oplus_i H^{2i}(M,\BQ) \to \oplus_i H^{2i}(M,\BQ)
\]
be the universal map that takes the exponential Chern character to Chern classes,
so in particular $c(E) = \ell( \ch(E) )$ for any vector bundle.
Then given $\alpha \in H^{\ast}(M)$ we write $c_m(\alpha)$ for $\left[ \ell(\alpha) \right]_{2m}$.
\end{rmk}

\begin{rmk}
In Theorem~\ref{thm:Markman_operator}, since the morphism $\gamma(g)$ is a ring isomorphism we have $\gamma(g)1 = 1$. Since $\gamma(g)$ preserves degree and is isometric,
it hence sends the class of a point on $M_1$ to the class of a point on $M_2$. For any $\sigma \in H^{\ast}(M_1)$ we thus observe hat
\[ \int_{M_1} \sigma = \int_{M_2} \gamma(g)(\sigma). \]
\end{rmk}

\begin{proof}[Proof of Theorem~\ref{thm:Markman_operator}]
If $g$ is an integral isometry, then the
statement of the theorem is a~combination
of Theorems~1.2 and~3.10 of~\cite{Markman}.
The proof is involved: Markman establishes that operators $\gamma(g)$ satisfying (a) and (b) exists
by considering arbitrary compositions of parallel transport operators and pushforwards by isomorphisms induced by auto-equivalences.
Then a~small computation starting from an expression for the diagonal class of $M_1$ in terms of the universal sheaf $\CF$ in~\cite{MarkmanGenerators},
shows that conditions $(a)$ and~$(b)$ for any homomorphism forces the expression~\eqref{Mopnew}.
Hence those homomorphisms are uniquely determined.
This last step holds even for homomorphisms defined over $\BC$ which satisfy $(a)$ and~$(b)$.

In the general case, one defines the operator $\gamma(g)$ by~\eqref{Mopnew}.
Then $(a)$ and~$(b)$ holds for a~Zariski dense subset of all operators $g$ (i.e., for the integral isometries).
Hence it holds for all~$g$. Then by the uniqueness statement one observes $(c)$. Again $(d)$ follows by the Zariski density argument from the integral case
(which is~\cite[Theorem~1.2(6)]{Markman}).
We also refer to~\cite[Proposition 5.1]{Frei} for more details on extending the Markman operator from integral isometries to isometries defined over more general coefficient rings.
\end{proof}
One can reinterpret the condition $(f \otimes g)(u_{v_1}) = u_{v_2}$ in terms of generators of the cohomology ring.
Following~\cite[equation~(3.23)]{Markman}, consider the canonical morphism
\[
B \colon \ H^{\ast}(S,\BQ) \to H^{\ast}(M,\BQ)
\]
defined by
\[
B(x) = \pi_{M \ast}\big( u_{v} \cdot x^{\vee} \big).
\]
We write $B_k(x)$ for its component in degree $2k$. In particular,
$B_0(x) = -(x,v)$ and
$B_1(x) = \theta_{\CF}(x)$ for all $x \in v^{\perp}$.

\begin{Lemma} 
Let $f \colon H^{\ast}(M_1, \BQ) \to H^{\ast}(M_2, \BQ)$ be a degree-preserving isometric ring isomorphism.
Then the following are equivalent:
\begin{enumerate}
\item[$(a)$] $(f \otimes g)(u_{v_1}) = u_{v_2}$,
\item[$(b)$] $f(B(x)) = B(gx)$ for all $x \in H^{\ast}(S_1,\BQ)$.
\end{enumerate}
\end{Lemma}

\begin{proof}
Since $g$ is an isometry of the Mukai lattice we have for $x \in H^{\ast}(S_1)$ the following equality in $H^{\ast}(M_2)$:
\[
\pi_{M_2 \ast}(u_{v_2} \cdot (gx)^{\vee}) = \pi_{M_2 \ast}\big( \big(1 \otimes g^{-1}\big) u_{v_2} \cdot x^{\vee} \big).
\]
{\samepage
Indeed, if we write $u_{v_2} = \sum_i a_i \otimes b_i$ under the K\"unneth decomposition, then
\begin{align*}
\pi_{M_2 \ast}\big( \big(1 \otimes g^{-1}\big)( u_{v_2} ) \cdot x^{\vee} \big)
& = \sum_i a_i \int_{S_1} g^{-1}(b_i) x^{\vee}
 = \sum_i -a_i \cdot \big(g^{-1}(b_i) \cdot x\big) \\
& = \sum_i -a_i \cdot (b_i \cdot g(x))  = \sum_i a_i \int_{S_2} b_i g(x)^{\vee} \\
& =  \pi_{M_2 \ast}\big( u_{v_2} \cdot g(x)^{\vee} \big).
\end{align*}}

Hence we see that:
\begin{align*}
(b) \ \Longleftrightarrow \ & \forall x \in H^{\ast}(S_1, \BZ) \colon \ f \pi_{M_1 \ast}\big(u_{v_1} \cdot x^{\vee}\big) = \pi_{M_2 \ast}\big( u_{v_2} \cdot (gx)^{\vee} \big) \\
\Longleftrightarrow\ & \forall x \in H^{\ast}(S_1, \BZ) \colon \ \pi_{M_2 \ast}\big( (f \otimes 1)u_{v_1} \cdot x^{\vee} \big) = \pi_{M_2 \ast}\big( \big(1 \otimes g^{-1}\big)u_{v_2} \cdot x^{\vee} \big) \\
\Longleftrightarrow\ & (f \otimes 1)(u_{v_1}) = \big(1 \otimes g^{-1}\big)(u_{v_2}) \\
\Longleftrightarrow\ & (a).
\tag*{\qed}
\end{align*}
\renewcommand{\qed}{}
\end{proof}

\begin{cor} \label{corX}
In the setting of Theorem~$\ref{thm:Markman_operator}$, $\gamma(g) B(x) = B(gx)$.
\end{cor}

\subsection{Universality}
We apply Theorem~$\ref{thm:Markman_operator}$ to study descendent integrals over $M$.
Let $k \geq 0$ and let $P(t_{ij}, u_r)$ be a~polynomial depending on the variables
\[
t_{j,i}, \quad j = 1, \dots, k, \quad i \geq 1, \qquad \text{and} \qquad u_r, \quad r \geq 1.
\]
Let also $A = (a_{ij})_{i,j=0}^{k}$ be a $(k+1) \times (k+1)$-matrix.

Our main result is the following.

\begin{thm}[universality] \label{thm:universality}
There exists $I(P,A) \in \BQ$ $($depending only on $P$ and $A)$ such that
for any $M = M_H(v)$ with $\dim(M)>2$ and for any $x_1, \dots, x_k \in \Lambda$ with
\begin{equation} \label{intersection matrix} \begin{pmatrix}
v \cdot v & (v \cdot x_i)_{i=1}^k \vspace{1mm}\\
(x_i \cdot v)_{i=1}^{k} & ( x_{i} \cdot x_j )_{i,j=1}^{k}
\end{pmatrix} = A \end{equation}
we have
\[ \int_{M} P( B_i(x_j), c_r(T_M) ) = I(P,A). \]
\end{thm}

In other words, the integral
\[ \int_{M} P( B_i(x_j) , c_r(T_M) ) \]
depends upon the above data only through $P$,
the dimension $\dim M = 2n$,
and the pairings $v \cdot x_i$ and $x_i \cdot x_j$ for all $i$, $j$.

The proof of Theorem~\ref{thm:universality} will proceed in several steps.
We begin with a general vanishing result.
\begin{prop} \label{prop:vanishing}
Let $M=M_H(v)$ be a moduli space of stable sheaves on $S$ of dimension $2n>2$,
and let $x_1, \dots, x_k$, $w \in \Lambda_{\BC}$ be given with $w \cdot y = 0$ for all $y \in \{ v, x_1, \dots, x_k, w \}$.
Then any integral of the form
\begin{equation} \label{monomial}
\int_{M} \prod_{i=1}^{\ell} B_{s_i}(w) \cdot (\textup{monomial in } B_i(x_j) \textup{ and } c_r(T_M))
\end{equation}
for some $s_i \in \BZ$ vanishes unless $\ell=0$.
\end{prop}

\begin{proof}
We give two proofs of this fact.
For the first proof, choose an isometry $g \colon \Lambda_{\BC} \to \Lambda_{\BC}$ such that
\[
g(v) = 1- (n-1) \pt, \qquad
w' := g(w) \in H^2(S,\BC) ,
\]
where $v \cdot v = 2n-2$. Such an isometry exists since $v \cdot v > 0$ and $\SO(\Lambda_{\BC})$ acts transitively on vectors of the same square.
By Theorem~\ref{thm:Markman_operator}$(a)$ for the first and Corollary~\ref{corX} and Theorem~\ref{thm:Markman_operator}$(d)$ for the second equation, we find that
\begin{gather*}
\int_{M} \prod_{i=1}^{\ell} B_{s_i}(w) \cdot \big(\textup{monomial in } B_i(x_j) \textup{ and } c_r(T_M)\big)
\\ \qquad
{}=\int_{S^{[n]}} \gamma(g)\bigg(\prod_{i=1}^{\ell} B_{s_i}(w) \cdot \big(\textup{monomial in } B_i(x_j) \textup{ and } c_r(T_{S^{[n]}}) \big) \bigg)
\\ \qquad
{}=\int_{S^{[n]}} \prod_{i=1}^{\ell} B_{s_i}(w') \cdot \big(\textup{monomial in } B_i(gx_j) \textup{ and } c_r(T_{S^{[n]}})\big).
\end{gather*}
By~\cite[Theorem 4.1]{EGL} (or more precisely, the induction method used in the proof),
this last integral depends upon $w'$ only through its intersection numbers against products of Chern classes of $S$ and degree-components of $g x_j$.\footnote{Since $w' \in H^2(S)$ we always have
$(w' \cdot [g x_j]_{k} ) = 0$ for $k=0,4$.
The vanishing in case $k=2$ follows from $(w' \cdot gx_j) = 0$.}
Since these intersections numbers are all zero, we may replace~$w'$ by $0$, in which case the claimed vanishing follows immediately.
\end{proof}

\begin{proof}[Alternative proof]
If $w=0$ there is nothing to prove, so let $w \neq 0$.
Choose $w' \in \Lambda_{\BC}$ such that $w \cdot w'=1$ and $w' \cdot w' = w' \cdot v = 0$.
Extend $v$, $w$, $w'$ to a basis $\{ v, w, w' \} \cup \{ e_i \}_{i=4}^{24}$ of $\Lambda_{\BC}$.
For any $j$, expand $x_j$ in this basis:
\[
x_j = a_1 v + a_2 w + a_3 w' + a_4 e_4 + \dots + a_{24} e_{24}.
\]
Because $x_j \cdot w=0$, we must have $a_3=0$.
By an induction on the number of classes $x_j$,
we know the claim of Proposition~\ref{prop:vanishing} if $x_j$ is a multiple of $w$.\footnote{Because
the term $B_i(x_j)$ can be moved to the product $\prod_{i=1}^{\ell} B_{s_i}(w)$ in \eqref{monomial}.}
Moreover, if we know the claim for $x_j \in \{ u_1, u_2 \}$ for some $u_1, u_2 \in \Lambda_{\BC}$ then we know it for $x_j = u_1 + u_2$
by expanding the monomial in \eqref{monomial}.
Hence we may replace $x_j$ by $x_j - a_2 w$. In other words, we may assume that $a_2=0$.
Doing so for all $j$, we hence see that
$w' \in \Lambda_{\BC}$ satisfies
\[
w' \cdot w = 1, \qquad w' \perp \mathrm{Span}(w', v, x_1, \dots, x_k).
\]

Consider the Lie algebra $\Fg = \mathfrak{so}\big(v^{\perp}\big) \cong \wedge^2 \big(v^{\perp}\big)$.
Theorem~\ref{thm:Markman_operator} induces a Lie algebra action $\gamma \colon \Fg \to \End H^{\ast}(M)$.
By Theorem~\ref{thm:Markman_operator}(a) $\gamma(\Fg)$ acts by derivations on $H^{\ast}(M)$
and acts trivially on~$H^{4n}(M)$.
(This Lie algebra action is part of the Looijenga--Lunts--Verbistky Lie algebra action, see~\cite[Lemma 4.13]{Markman}.)
Take $w \wedge w' \in \Fg$. Since the Lie algebra acts trivial on~$H^{4n}(M)$ we have
\begin{equation*}
\int_{M} \gamma(w \wedge w') \bigg( \prod_{i=1}^{\ell} B_{s_i}(w) \cdot \big( \textup{monomial in } B_i(x_j) \textup{ and } c_r(T_M) \big) \bigg) = 0.
\end{equation*}
On the other hand, by Corollary~\ref{corX} we have
$\gamma(w \wedge w') B_{s_i}(w) = B_{s_i}(w)$ and $\gamma(w \wedge w') B_i(x_j) = 0$,
and by Theorem~\ref{thm:Markman_operator}$(d)$ we have $\gamma(w \wedge w') c_r(T_{M}) = 0$.
Since $\gamma(w \wedge w')$ acts by derivations, we~also~get\pagebreak
\begin{gather*}
\int_{M} \gamma(w \wedge w') \bigg( \prod_{i=1}^{\ell} B_{s_i}(w) \cdot \big( \textup{monomial in } B_i(x_j) \textup{ and } c_r(T_M) \big) \bigg)
\\ \qquad
{}= \ell \cdot \int_{M} \prod_{i=1}^{\ell} B_{s_i}(w) \cdot \big( \textup{monomial in } B_i(x_j) \textup{ and } c_r(T_M) \big). \tag*{\qed}
\end{gather*}
\renewcommand{\qed}{}
\end{proof}

\begin{Lemma} \label{lemma:replace with non-degenerate}
In the situation of Theorem~$\ref{thm:universality}$, there exists $y_i \in \Lambda_{\BC}$ which have the same intersection matrix as in \eqref{intersection matrix},
satisfy
\[ \int_{M} P( B_i(x_j), c_r(T_M) ) = \int_{M} P( B_i(y_j), c_r(T_M) ) \]
and such that the span $L = \mathrm{Span}( v, y_1, \dots, y_k ) \subset \Lambda_{\BC}$ is non-degenerate $($i.e., the restriction of the inner product of $\Lambda_{\BC}$ onto $L$ is non-degenerate$)$.
\end{Lemma}

\begin{proof}
Let $L = \mathrm{Span}(v, x_1, \dots, x_k)$.
Assume that $L$ is degenerate, i.e., there exists a non-zero $w \in L$ such that $w \cdot x_i = 0$ for all $i$ and $w \cdot v=0$.
Since $v \cdot v \geq 2$, we have that $v$, $w$ are linearly independent. Hence they can be extended to a basis $u_0, \dots, u_d$ of $L$ with $u_0 = w$ and $u_1=v$.
For every $i$ let $\lambda_i \in \BC$ be the unique scalar such that
\[
x_i - \lambda_i w \in \mathrm{Span}( u_1, \dots, u_d ).
\]
We hence obtain
\begin{gather*}
\int_{M} P( B_i(x_j), c_r(T_M) ) =
\int_{M} P( B_i(x_j - \lambda_j w) + B_i(w), c_r(T_M) )
\\[1ex] \qquad\qquad\qquad\quad\,
{} \overset{\rm Proposition~\ref{prop:vanishing}}{=}
\int_{M} P( B_i(x_j - \lambda_j w), c_r(T_M) ).
\end{gather*}
Set $y_j = x_j - \lambda_j w$. If $\mathrm{Span}(v, y_1, \dots, y_k)$ is non-degenerate, we are done,
otherwise repeat the above process. This process has to stop, since the dimension of the span drops by one in each step.
\end{proof}

We also require two basic linear algebra lemmata:
\begin{Lemma} \label{lemma:la1}
Let $V$ be a finite-dimensional $\BC$-vectorspace with a $\BC$-linear inner product.
Let $v_1, \dots, v_k \in V$ be a list of vectors with Gram matrix
\[
g = \big( g_{ij} \big)_{i,j=1}^{k}, \qquad g_{ij} = \langle v_i, v_j \rangle.
 \]
Then $\mathrm{rank}(g) \leq \dim( \mathrm{Span}(v_1, \dots, v_k))$.
If moreover $\mathrm{Span}(v_1, \dots, v_k)$ is a non-degenerate subvectorspace of $V$, then
$\mathrm{rank}(g) = \dim( \mathrm{Span}(v_1, \dots, v_k))$.
\end{Lemma}
\begin{proof}
Let $w_1, \dots, w_{\ell} \in V$ be a list of vectors such that $h_{ij}=\langle w_i, w_j \rangle$ is invertible.
Pairing any linear relation between the $w_i$'s with $w_j$ for $j=1,\dots, \ell$,
and multiplying this system of equations by the inverse of $h$
shows that the $w_1, \dots, w_{\ell}$ are linearly independent.
This proves the first claim.
For the second claim, we can choose a subset $\{ w_1, \dots, w_d \} \subset \{ v_1, \dots, v_k \}$ which forms a basis of $L = \mathrm{Span}(v_1, \dots, v_k)$
and observe that the matrix of the isomorphism $L \to L^{\vee}$ induced by the inner product
with respect to the basis $\{ w_i \}$ and the dual basis $\{ w_i^{\ast} \}$
is the Gram matrix of the $w_i$.
This shows that $\mathrm{rank}(g) \geq \dim L$.
\end{proof}
\begin{Lemma} \label{lemma:a2}
Let $V$ be a finite-dimensional $\BC$-vectorspace with a $\BC$-linear inner product.
Let $v_1, \dots, v_k \in V$ and $w_1, \dots, w_k \in V$ be lists of vectors such that
\begin{itemize}\itemsep=0pt
\item[$(i)$] $L = \mathrm{Span}(v_1, \dots, v_k)$ is non-degenerate,
\item[$(ii)$] $M = \mathrm{Span}(w_1, \dots, w_k)$ is non-degenerate,
\item[$(iii)$] $\langle v_i, v_j \rangle = \langle w_i, w_j \rangle$ for all $i,j$.
\end{itemize}
Then there exists an isometry $\varphi \colon V \to V$ such that $\varphi(v_i) = w_i$ for all $i$.
\end{Lemma}

\begin{proof}
By Lemma~\ref{lemma:la1} and assumptions $(i)$ and $(ii)$ we know that
\[
\dim L = \mathrm{rank} (\langle v_i, v_j \rangle)_{i,j=1}^{k} =
\mathrm{rank} (\langle w_i, w_j \rangle)_{i,j=1}^{k} = \dim M.
\]
Choose a basis of $L$ from the $v_1, \dots, v_k$, which we can assume is of the form $v_1, \dots, v_d$, where $d=\dim(L)$.
By assumption $(i)$ and Lemma~\ref{lemma:la1} the gram matrix $G := (\langle v_i, v_j \rangle )_{i,j=1}^{d}$ is invertible.
But $G$ is also the Gram matrix of $w_1, \dots, w_d$ by assumption $(iii)$,
so the same lemma implies that $w_1, \dots, w_d$ is linearly independent and hence a basis of $M$.
Define an isometry
\[
\varphi \colon\ V \to V
\]
by setting $\varphi(v_i) = w_i$ for $i=1,\dots, d$,
and by letting $\varphi_{L^{\perp}} \colon L^{\perp} \to M^{\perp}$ be an arbitrary isometry.
It remains to show that $\varphi(v_i) = w_i$ for $i=d+1, \dots, k$.
For this observe that for any $v \in L$ we have\looseness=-1
\begin{gather*}
v = \sum_{a=1}^{d} \langle v, v_a \rangle \big(G^{-1}\big)_{ab} v_b
\end{gather*}
and similarly for any $w \in M$.
The claim hence follows by writing every $v_i$ in this form, applying~$\varphi$ and using assumption $(iii)$.
\end{proof}

We are ready to prove Theorem~\ref{thm:universality}.
\begin{proof}
Let $(M(v), x_i)$ and $(M(v'), x_i')$ be two pairs with the same intersection matrix $A$.
By~Lemma~\ref{lemma:replace with non-degenerate}, we may assume that $v,x_1, \dots, x_k$ and
$v', x_1', \dots, x_k'$ span a non-degenerate subspace of $\Lambda_{\BC}$.
Hence, by Lemma~\ref{lemma:a2}, there exists an isometry
\[
g \colon\ H^{\ast}(S,\BC) \to H^{\ast}(S', \BC)
\]
which takes $(v, x_1, \dots,x_k)$ to $(v', x'_1, \dots, x_k')$.
We find that
\begin{align*}
\int_{M(v)} P(B_i(x_j), c_r(T_{M(v)}))
& \overset{\textup{(Theorem~\ref{thm:Markman_operator})}}{=} \int_{M(v')} \gamma(g) P(B_i(x_j), c_r(T_{M(v')})) \\
& \overset{\textup{(Corollary~\ref{corX})}}{=}
\int_{M(v')} P( B_i(g x_j), c_r(T_{M(v')})) \\
&\qquad =
\int_{M(v')} P( B_i(x_j'), c_r(T_{M(v')})).\tag*{\qed}
\end{align*}
\renewcommand{\qed}{}
\end{proof}

\subsection{Case of dimension 2} \label{subsec:Case dim 2}
We discuss how to evaluate integrals
\begin{equation} \label{descendent integral repeated}
\int_{M} \tau_{k_1}(\gamma_1) \cdots \tau_{k_{\ell}}(\gamma_{\ell}) P(c_r(T_{M})),
\end{equation}
whenever $M = M_H(v)$ is a $2$-dimensional moduli space of stable sheaves, and hence a $K3$ surface.
The universal family\footnote{If only a twisted universal family exists, then
we have an equivalence to the derived category of twisted sheaves on $M$ with the corresponding twist, see~\cite{HuybrechtsK3}.} $\CF$ in this case induces a derived auto-equivalence
\[
\Phi\colon \  D^b(S) \to D^b(M), \qquad
\CE \mapsto \pi_{M \ast}( \pi_S^{\ast}(\CE) \otimes \CF ).
\]
The induced action on cohomology
\[
\Phi_{\ast}\colon\  H^{\ast}(S,\BZ) \to H^{\ast}(M, \BZ),
\qquad \gamma \mapsto \pi_{M \ast}( v(\CF) \cdot \pi_S^{\ast}(\gamma))
\]
defines an isometry of Mukai lattices (in fact, a Hodge isometry), see~\cite[Chapter 16]{HuybrechtsK3} for references for these well-known facts.

We specialize to the case where $\rk(v) > 0$,
which is the only one we consider in the applications.
Consider the normalized action
\[
\widetilde{\Phi}\colon\  H^{\ast}(S,\BQ) \to H^{\ast}(M, \BQ), \qquad
\gamma \mapsto \pi_{M \ast}\big( {\rm e}^{-c_1(\CF)/\rk(v)} v(\CF) \cdot \pi_S^{\ast}(\gamma)\big).
\]
Let us write $\ch(\CF) = \rk(v) + \pi_M^{\ast}(\ell) + \pi_S^{\ast}(c_1(v)) + (\ldots)$, where $\ldots$ stands for terms of degree $\geq 4$.
Then we have $\widetilde{\Phi} = {\rm e}^{-\ell/\rk(v)} \Phi_{\ast}({\rm e}^{-c_1(v)/\rk(v)} \cup ( - ) )$
which shows that $\widetilde{\Phi}$ is still a Hodge isometry.
Using the fact that $\widetilde{\Phi}$ is a Hodge isometry implies\footnote{By direct computation, the degree
zero component of $\widetilde{\Phi}(\gamma)$ is $\rk(v) \int_S \gamma$.
Then observe the degree $1$ term of~$\widetilde{\Phi}(\pt)$ vanishes by construction of $\widetilde{\Phi}$. Hence, $\big(\widetilde{\Phi}(\pt), \widetilde{\Phi}(\pt)\big)=0$
shows the first line. The others follow similarly.}
\begin{gather}
\widetilde{\Phi}(\pt)  = \rk(v),\qquad
\widetilde{\Phi}(L)  = \varphi(L),\qquad
\widetilde{\Phi}(1) = \frac{1}{\rk(v)} \pt, \label{phi tilde}
\end{gather}
where $\varphi\colon H^2(S,\BQ) \to H^2(S,\BQ)$ is a Hodge isometry.

\subsection{Proof of Theorem~\ref{thm:intro main}}

If $\dim M > 2$, the claim follows by Theorem~\ref{thm:universality} since
$(a)$ any descendent $\tau_k(\gamma)$ defined as in \eqref{defn descendent}
can been written as a polynomial in classes $B_j(x)$,
and $(b)$ for any list of vectors $v, x_1, \dots, x_k \in \Lambda_{\BC}$
after an isometry of $\Lambda_{\BC}$ we may assume that $v$ is the Mukai vector which defines the Hilbert scheme of $n$ points on a $K3$ surface.\footnote{By Eichler's criterion~\cite[Lemma~7.5]{GHS}, this isometry can be defined over the integers.}

If $\dim M = 2$ and $\rk(v) > 0$, as discussed in Section~\ref{subsec:Case dim 2} any descendent $\tau_{k_i}(\gamma)$
can be written in terms of polynomials in classes $\widetilde{\Phi}(\alpha)$, where $\alpha$ is effectively determined by $\gamma$.
Since any integral \eqref{descendent integral repeated} can involve
at most two classes of positive degree, this integral can be written
as linear combination of the Mukai pairing between classes $\widetilde{\Phi}(\alpha)$ and $\widetilde{\Phi}(\alpha')$ for various $\alpha$, $\alpha'$.
Since~$\widetilde{\Phi}$ is a Hodge isometry, these are just the Mukai pairings between $\alpha$ and $\alpha'$.
This effectively determines the integrals \eqref{descendent integral}.
We also refer to Section~\ref{subsec:GK conjecture dim 2} for a concrete implementation of this algorithm.

The case where $\dim(M) = 2$ and $\rk(v) = 0$ is similar to the $\dim M = 2$, $\rk(v)>0$ case, and left to the reader.
\qed

\section{The G\"ottsche--Kool conjecture} \label{GK conjecture}

Let $S$ be a $K3$ surface and let $M$ be a proper fine $2n$-dimensional moduli space of stable sheaves on $S$ of Mukai vector $v$. Let $\CF$ be a universal family. We assume that $\rk(v) > 0$.
Our goal is to show that
for any $\alpha \in K(S)$,
class $L \in H^2(S)$ and $u \in \BC$ we have
\begin{equation*}
\int_{M} c(\alpha_M) {\rm e}^{\mu(L) + u \mu(\pt)}
=\int_{S^{[n]}} c(\beta_{S^{[n]}}) {\rm e}^{\mu(L) + u \rk(v) \mu(\pt)},
\end{equation*}
where $\beta \in K(S)$ is as specified in Theorem~\ref{thm:Intro}.

In Section~\ref{subsec:GK conjecture dim 2}, we first tackle the case $\dim M = 2$ separately,
and then afterwards prove the $\dim M > 2$ case.

\subsection[Proof of Theorem~1.2 in case dim(M)=2]
{Proof of Theorem~\ref{thm:Intro} in case $\boldsymbol{\dim(M)=2}$}\label{subsec:GK conjecture dim 2}
Observe that $S^{[1]} \cong S$, and for $\beta \in K(S)$ and $L \in H^2(S)$ we have
\[
\beta_{S^{[1]}} = \ch(\beta) - \chi(\beta), \qquad
\mu_{S^{[1]}}(L) = L \in H^2(S), \qquad
\mu_{S^{[1]}}(\pt) = \pt.
\]
Hence we need to prove
\begin{equation}
\label{thm eqn 3}
\int_{M} c(\alpha_M) {\rm e}^{\mu(L) + u \mu(\pt)}
=
\int_{S} c(\beta) {\rm e}^{L + u \rk(v) \pt}.
\end{equation}
Recall from Section~\ref{subsec:Case dim 2} the Hodge isometry $\widetilde{\Phi} \colon H^{\ast}(S,\BQ) \to H^{\ast}(M,\BQ)$ defined by the universal family $\CF$.
By comparing the definition of $\alpha_M$ and $\mu(\sigma)$ with the correspondence defining $\widetilde{\Phi}$ we find
\begin{gather*}
\alpha_M  = -\frac{1}{\sqrt{\td_M}} \widetilde{\Phi}(v(\alpha)),
\\
\mu_M(\sigma)  = \biggl[ - \frac{1}{\sqrt{\td_M}} \widetilde{\Phi}\big(\sigma/\sqrt{\td_S}\big) \biggr]_{\deg(\sigma)}.
\end{gather*}
In particular, by \eqref{phi tilde} we have $\mu_M(L) = -\widetilde{\Phi}(L)$.
Using \eqref{phi tilde} we obtain
\begin{gather*}
\int_{M} \mu_M(\pt) = \int_{M} -(1-\pt) \rk(v) \cdot 1 = \rk(v),
\\
\int_M \mu_M(L)^2 = \int_{M} \big({-} \widetilde{\Phi}(L)\big)^2 = \big(\widetilde{\Phi}(L), \widetilde{\Phi}(L)\big) = (L,L) = \int_S L^2,
\\
\int_{M} c_1(\alpha_M) \mu_M(L) = \int_{M} \alpha_M \cup \big({-} \widetilde{\Phi}(L)\big) = \int_{M} \widetilde{\Phi}(v(\alpha)) \cdot \widetilde{\Phi}(L)= \big(\widetilde{\Phi}(v(\alpha)), \widetilde{\Phi}(L)\big)
\\ \hphantom{\int_{M} c_1(\alpha_M) \mu_M(L)}
{} = (v(\alpha), L)= \int_{S} c_1(\alpha) \cdot L.
\end{gather*}
Using \eqref{phi tilde} again we moreover have
\begin{align*}
\alpha_M &= -(1-\pt) \widetilde{\Phi}( \rk(v) + c_1(\alpha) + v_2(\alpha) )
\\
&= - \rk(v) \int_{S} v_2(\alpha) - \varphi(c_1(\alpha)) + \biggl( - \frac{\rk(\alpha)}{\rk(v)} + \rk(v) \int_S v_2(\alpha) \biggr) \pt,
\end{align*}
and hence (with $\alpha_{M,k}$ be the degree $2k$ component of $\alpha_M$) we get
\[ \int_{M} c_2(\alpha_M) = \int_{M} - \alpha_{M,2} + \frac{\alpha_{M,1}^2}{2} = \frac{\rk(\alpha)}{\rk(v)} - \rk(v) \int_S v_2(\alpha) + \frac{c_1(\alpha)^2}{2}. \]
By inspection one sees now that if $\beta$ satisfies \eqref{conditions on beta},
then equation \eqref{thm eqn 3} holds.
This completes the proof. \qed

\subsection{Comparing normalizations} 
From now on assume that
\[ \dim M > 2. \]
Let $\alpha \in K(S)$ and consider the definition of $\alpha_M$
using the Grothendieck--Riemann--Roch formula:
\begin{gather*}
\alpha_M = - \pi_{M \ast}\bigg( v(\alpha) \ch(\CF) \sqrt{\td_S}
\exp\biggl( - \frac{c_1(\CF)}{\rk(v)} \biggr) \bigg).
\end{gather*}

The class $\alpha_M$ is easily expressed in terms of Markman's normalization:
\begin{Lemma} We have
\[
\alpha_M=
- B\bigg( v\big(\alpha^{\vee}\big) \exp\bigg( \frac{c_1(v)}{ \rk(v)} \bigg) \bigg)
\exp\bigg(  B_1\bigg( \frac{-\pt}{\rk(v)} - \frac{v}{v \cdot v} \bigg) \bigg).
\]
\end{Lemma}

\begin{proof}
Using that $\Pic(M \times S) = \Pic(M) \oplus \Pic(S)$ we can write
\[ c_1(\CF) = \pi^{M \ast}(\ell) + \pi_S^{\ast}(c_1(v)) \]
for some $\ell \in H^2(M)$. By calculating $\theta_{\CF}(\pt)$ one finds
$\ell = \theta_{\CF}(\pt)$.
Hence
\begin{align*}
\alpha_M
& = - \pi_{M \ast}\bigg(v(\alpha)\ch(\CF)\sqrt{\td_S}\exp\biggl(- \frac{c_1(v)}{ \rk(v)}\biggr)\bigg)
\exp( \theta_{\CF}(\pt) / (\pt \cdot v) )
\\
& = - B\bigg( v\big(\alpha^{\vee}\big) \exp\bigg( \frac{c_1(v)}{ \rk(v)} \bigg) \bigg)
\exp\bigg(  B_1\bigg( \frac{-\pt}{\rk(v)} - \frac{v}{v \cdot v} \bigg) \bigg) .\tag*{\qed}
\end{align*}
\renewcommand{\qed}{}
\end{proof}

For $\sigma \in H^{\ast}(S)$ recall also the class
\[
\mu(\sigma)=- \pi_{M \ast}\big(\ch_2\big(\CF \otimes \det(\CF)^{-1/\rk(v)}\big) \pi_S^{\ast}(\sigma)\big)
\]
(defined by the GRR expression if only a semi-universal family exists).

\begin{Lemma} If $\sigma \in H^{\ast}(S)$ is homogeneous, then $\mu(\sigma)$ is the component of degree $\deg(\sigma)$ of
\[
- \exp\bigg(B_1\bigg( \frac{p}{p \cdot v} - \frac{v}{v \cdot v} \bigg) \bigg)
B\bigg( \sigma^{\vee} \exp\bigg( \frac{c_1(v)}{\rk(v)} \bigg) \sqrt{\td_S}^{-1} \bigg).
\]
\end{Lemma}

\begin{proof}
We have that $\mu(\sigma)$ is the degree $\deg(\sigma)$ component of
\begin{gather*}
- \pi_{M \ast} \big(\ch\big( \CF \otimes \det(\CF)^{-1/\rk(v)} \big) \pi_S^{\ast}(\sigma)\big)
=- \pi_{M \ast} \big(\ch( \CF) \exp( - c_1(\CF) / \rk(v) ) \pi_S^{\ast}(\sigma) \big)
\\ \quad
{}= - \exp\bigg( \frac{\theta_{\CF}(\pt)}{\pt \cdot v}\! -\! \frac{\theta_{\CF}(v)}{v \cdot v} \bigg)
\exp\bigg( \frac{\theta_{\CF}(v)}{v \cdot v} \bigg)
\pi_{M \ast}\big( \ch(\CF) \pi_S^{\ast}\big( \sigma^{\vee} {\rm e}^{c_1(v)/\rk(v)} \sqrt{\td_S}^{-1} \big)^{\vee}\! \sqrt{\td_S} \big)
\\ \quad
{}=  - \exp\bigg( B_1\bigg( \frac{p}{p \cdot v} - \frac{v}{v \cdot v} \bigg) \bigg)
B\bigg( \sigma^{\vee} \exp\bigg( \frac{c_1(v)}{\rk(v)} \bigg) \sqrt{\td_S}^{-1} \bigg),
\end{gather*}
where we used again $c_1(\CF) = \pi_{M}^{\ast} \theta_{\CF}(\pt) + \pi_S^{\ast} c_1(V)$.
\end{proof}

In particular, for $L \in H^2(S)$ we have that
\begin{align*}
\mu(L) & =
B_1\bigg(L \exp\bigg(\frac{c_1(v)}{\rk(v)}\bigg)\bigg) - B_1\bigg(\frac{p}{p \cdot v}
- \frac{v}{v\cdot v} \bigg)
\end{align*}
and that $\mu(\pt)$ is a polynomial in
$B_1\big(\frac{p}{p \cdot v} - \frac{v}{v \cdot v}\big)$ and $B_i(\pt)$.

\subsection{Dependence}
By Theorem~\ref{thm:universality} we conclude that any integral
\begin{equation}
\int_{M} P(\alpha_{M,k}, \mu(L), \mu(u \pt))
\label{integral}
\end{equation}
(such as the Segre number) only depends upon $P$ and
the intersection pairings in the Mukai lattice of the classes
\begin{equation} v, \quad
\pt/\rk(v), \quad
v(\alpha)^{\vee} \exp\left( \frac{c_1(v)}{\rk(v)} \right),
\quad L  \exp\left( \frac{c_1(v)}{\rk(v)} \right),
\quad u \pt. \label{dependence parameters} \end{equation}

Explicitly, the interesting pairings for the first three classes are
\begin{align*}
\text{(i)} \quad   & v \cdot v(\alpha)^{\vee} \exp\bigg( \frac{c_1(v)}{\rk(v)}\bigg)
= - v_2(\alpha) \cdot \rk(v) + \frac{1}{2} \frac{\rk(\alpha)}{\rk(v)} (v \cdot v),
\\
\text{(ii)} \quad   & \pt/ \rk(v) \cdot v(\alpha)^{\vee} \exp\bigg( \frac{c_1(v)}{\rk(v)} \bigg)
= - \frac{\rk(\alpha)}{\rk(v)},
\\
\text{(iii)} \quad   & \bigg(v(\alpha)^{\vee} \exp\bigg( \frac{c_1(v)}{\rk(v)}\bigg)\bigg)^2 = v(\alpha) \cdot v(\alpha).
\end{align*}
The interesting intersections involving $L$ are
\begin{align*}
\text{(iv)} \quad   & v \cdot L  \exp\bigg( \frac{c_1(v)}{\rk(v)} \bigg)
= L \cdot c_1(v) - L \cdot c_1(v) = 0,
\\
& v(\alpha)^{\vee} \exp\bigg( \frac{c_1(v)}{\rk(v)} \bigg)\cdot
L  \exp\bigg( \frac{c_1(v)}{\rk(v)} \bigg)= v(\alpha)^{\vee} \cdot L = -c_1(\alpha) \cdot L,
\\
& \bigg( L  \exp\bigg( \frac{c_1(v)}{\rk(v)} \bigg) \bigg)^2 = L^2.
\end{align*}
The pairings with $u \pt$ are
$u \rk(v)$ times the pairings with $\pt/\rk(v)$.

\subsection{Moving to the Hilbert scheme}
Since \eqref{integral} only depends on the intersection pairings of \eqref{dependence parameters} we have that
\[
\int_{M} P(\alpha_{M,k}, \mu(L), \mu(u \pt))
=
\int_{S^{[n]}} P(\beta_{S^{[n]},k}, \mu(L), \mu(u' \pt))
\]
for any $K$-theory class $\beta \in K(S)$
and $u'\in \BC$ such that the list
\begin{equation}
\label{dep pars 2}
1 - (n-1)\pt, \qquad
\pt, \qquad
v(\beta)^{\vee},
\qquad L,
\qquad u' \pt
\end{equation}
has the same intersection numbers as the list~\eqref{dependence parameters}.
(The list \eqref{dep pars 2} is obtained from~\eqref{dependence parameters} by specializing to
$v=1-(n-1)\pt$, the Mukai vector of $S^{[n]}$.)

The interesting parts of the intersections of \eqref{dep pars 2} are
\begin{align*}
\text{(i)} \quad   & v \cdot v(\beta)^{\vee} = -v_2(\beta) + \frac{1}{2} \rk(\beta) (2n-2), \\
\text{(ii)} \quad   & \pt \cdot v(\beta)^{\vee} = - \rk(\beta), \\
\text{(iii)} \quad  & v(\beta)^{\vee} \cdot v(\beta)^{\vee} = v(\beta) \cdot v(\beta), \\
\text{(iv)} \quad   & v(\beta)^{\vee} \cdot L = -c_1(\beta) \cdot L.
\end{align*}
Equating (i)--(iv) for $M$ and $S^{[n]}$ we hence get the system
\begin{gather*}
- v_2(\alpha) \cdot \rk(v) + \frac{1}{2} \frac{\rk(\alpha)}{\rk(v)} (v \cdot v)  = -v_2(\beta) + \frac{1}{2} \rk(\beta) (2n-2),
\\
- \frac{\rk(\alpha)}{\rk(v)} =  - \rk(\beta),
\qquad
v(\alpha) \cdot v(\alpha)  = v(\beta) \cdot v(\beta),
\qquad
-c_1(\alpha) \cdot L  = -c_1(\beta) \cdot L.
\end{gather*}
Since $v(\alpha)^2 = c_1(\alpha)^2 - 2 \rk(\alpha) v_2(\alpha)$,
this is equivalent to the system:
\begin{equation} \label{final system}
\rk(\beta)  = \frac{ \rk(\alpha) }{ \rk(v) }, \quad \
v_2(\beta)  = \rk(v) v_2(\alpha), \quad \
c_1(\alpha)^2  = c_1(\beta)^2, \quad \
c_1(\alpha) \cdot L  = c_1(\beta) \cdot L.
\end{equation}
Moreover, we must have
\[ -u' =  u' \pt \cdot (1 - (n-1) \pt) = u \pt \cdot v = - \rk(v) u. \]

We have proven the following (which immediately implies
Theorem~\ref{thm:Intro}):
\begin{thm} For any polynomial $P$, we have
\[
\int_{M} P(\alpha_{M,k}, \mu(L), \mu(u \pt))
=
\int_{S^{[n]}} P(\beta_{S^{[n]},k}, \mu(L), \mu(u \rk(v) \pt))
\]
for any $K$-theory class $\beta \in K(S)$ such that
\eqref{final system} is satisfied.
\end{thm}

\subsection*{Acknowledgements}
I~thank Martijn Kool for bringing Conjecture~5.1 of~\cite{GK} to my attention
and Lothar G\"ottsche for useful comments.
I~am also indebted to the referees for careful reading of the paper and pointing out several issues in an earlier version.
The author is partially funded by the Deutsche Forschungsgemeinschaft (DFG) -- OB~512/1-1.


\pdfbookmark[1]{References}{ref}
\LastPageEnding

\end{document}